\pdfoutput=1
\documentclass[a4paper, 10pt]{article}
\usepackage{authblk, setspace, subfig, caption}
  \usepackage{url}
 \usepackage{amsmath, amssymb,amscd} %% matematika
 \usepackage{enumerate} %% seznami
 \usepackage[active]{srcltx} %% iskanje v izvorni datoteki
 \usepackage{amsthm, leftidx}
  \usepackage{multirow, booktabs}
\usepackage{algorithm}% http://ctan.org/pkg/algorithms
\usepackage{algpseudocode}% http://ctan.org/pkg/algorithmicx

\usepackage{graphicx}

\usepackage{amsthm} %% izreki...

\theoremstyle{definition}

\newtheorem{ex}{Example}

\theoremstyle{plain}

\newtheorem{prop}{Proposition}
\newtheorem{cor}{Corollary}

\newcommand{\RR}{\mathbb{R}}
\newcommand{\RRpo}{\mathbb{R}_{\ge 0}}

\newcommand{\NN}{\mathbb{N}}

\def\wset{\mathcal W}
\def\tset{\mathcal T}

\def\states{\mathcal X}

\usepackage[square,sort,comma,numbers]{natbib} %% literatura  
\usepackage[pdftex]{hyperref}

 \newcommand{\low}[1]{{\underline{#1}}}
 \newcommand{\up}[1]{{\overline{#1}}}

\title{Random walks on graphs with interval weights and precise marginals}
\author{Damjan \v{S}kulj \\ University of Ljubljana, Slovenia \\ damjan.skulj@fdv.uni-lj.si}

\begin{document}

\maketitle

\begin{abstract}
%points:
%- random walks on graphs correspond to reversible MC
%- uncertainty in parameters is often modelled with intervals
%- well established theory of imprecise Markov chains exists
%- it does not allow modelling reversible IMC
%- we propose a model of imprecise reversible MC
%- backward induction not applicable in our model
%- problem is not convex
%- we propose a local maximisation algorithm 
%- proves efficient, although the problem of multiple local extreme persists 

We propose a model of random walks on weighted graphs where the weights are interval valued, and connect it to reversible imprecise Markov chains. 
While the theory of imprecise Markov chains is now well established, this is a first attempt to model reversible chains. 
In contrast with the existing theory, the probability models that have to be considered are now non-convex. 
This presents a difficulty in computational sense, since convexity is critical for the existence of efficient optimization algorithms used in the existing models.  
The second part of the paper therefore addresses the computational issues of the model. 
The goal is finding sets of weights which maximize or minimize expectations corresponding to multiple steps transition probabilities.
In particular, we present a local optimization algorithm and numerically test its efficiency. 
We show that its application allows finding close approximations of the globally best solutions in reasonable time. 
	
%	We examine random walks on weighted undirected graphs with interval weights. They can be modelled by a new family of reversible imprecise Markov chains. Finding bounds of distributions over the set of vertices after several time steps is a high dimensional optimization problem. We therefore keep the model as simple as possible by fixing the total sum of weights of edges incident to a vertex constant in time and precisely known. The main focus is on computing expectation bounds of real valued maps with respect to imprecise probability distributions corresponding to the chain at consecutive steps. As the main result we propose a quick and efficient algorithm that computes local extrema in the neighbourhood of given initial solutions. While there is no apparent method to identify global extrema directly, the local optimisation is applied to a sample of initial solutions among which the best one is selected as a likely candidate for the global optimum. The method is tested on several randomly generated examples of graphs with various numbers of vertices for random walks of various lengths. It is confirmed that the local optimization algorithm greatly improves the efficiency of the optimization.
	
	\medskip
	\centerline{\bfseries Keywords}
	weighted graph, random walk, Markov chain, imprecise Markov chain, reversible Markov chains, local optimization, global optimization
\end{abstract}

\section{Introduction}
\subsection{Modelling uncertainty in Markov chains and weighted graphs}
Markov chains with the property that every sequence of states is equally likely no matter whether the process runs forwards or backwards are said to be \emph{reversible}. Reversible Markov chains are often interpreted and modelled with \textit{random walks} on weighted graphs (\cite{abe2014, aldous-fill-2014, aldous1989lower, Feige1995,  feige1995tight, gobel1974random,  lovasz1993random}) where the states of the chain are the vertices of the graph and transition probabilities are proportional to the weights of the edges incident to the initial vertex. Reversible Markov chains are often used in Monte Carlo methods (\cite{green1995reversible, hastings1970monte, jerrum1996markov}). Random walks on graphs have become very popular in network analysis (\cite{coppersmith1993random, lin2014mean, liu2012probabilistic,  zhang2014effects}), social networks (\cite{ backstrom2011supervised, li2011link, raey, yin2010unified}) and web recommender systems (\cite{fouss2007random}).

Modelling real world phenomena with Markov chains requires estimating a large number of parameters. Even with ever growing amounts of data at disposal this task is often impossible to achieve without serious uncertainty in the estimates. Ignoring this fact and regarding the parameters as precise leads to overprecise unreliable results. 
The need for more robust models for probability has led to various models known under the common name as \emph{theory of imprecise probabilities} (\cite{augustin2014introduction}). In particular, for Markov chains the theory of \emph{imprecise Markov chains} has been developed for discrete (\cite{decooman-2008-a, hart:98, skulj:09}) as well as continuous case (\cite{skulj2015165}). Most of the existing models are based on the theory of lower previsions (\cite{Miranda2008}). 

In the core of the theory of imprecise Markov chains is the idea that transition probabilities at each step are modelled with convex sets of probability distributions rather than single transition probabilities. Equivalently, all relevant probability distributions can then be modelled by non-additive functionals called \emph{coherent lower previsions}, which are defined as lower envelopes of sets of additive functionals. 

Weights in graphs often also reflect some relation between vertices obtained on the basis of imperfect data. One way of expressing the resulting uncertainty is to use intervals instead of precise weights. While being a compelling generalisation, the related optimisation problems seem to be generally hard (\cite{aissi2009min, averbakh2004interval}). Up until now finding minimum spanning tree and shortest paths in graphs with weighted intervals have received a lot attentions, while random walks seem to have not yet been explored. The lack of appropriate models of imprecise Markov chains and apparent high complexity of the general model might be among the reasons for this. The high complexity is also the main reason for our decision to keep our model simple by not allowing weights to vary completely freely within interval bounds, but instead assuming the sum of weights of edges incident to a given vertex to be constant. This could only be efficiently achieved by allowing loops, which then contain the non-allocated weight mass. 

\subsection{Model}
The aim of the present article is to extend the theory of imprecise Markov chains for the case of reversible chains; more specifically, random walks on weighted graphs with interval weights. Interval weights are interpreted as sets containing the precise weights that will actually set the probabilities of transitions. We also assume that weights are not constant in time but rather at every time step an unknown mechanism selects new set of weights, for which, except that they belong to the given intervals, we have no information available. Once the weights at certain time step are selected, transitions are calculated in the usual way. 
In our model we restrict the set of weights by requiring that total sum of weights of edges incident to a vertex is constant and precisely known. This is achieved by assigning the remaining mass to the loops (i.e. edges connecting the same vertices). This restriction will allow an efficient local optimisation for calculation of multiple steps probability bounds. Actually a similar effect is the result of the rate of leaving a state when modelling continuous time Markov chains. Having precisely given marginals while dependencies are imprecise is not that uncommon since usually there is a lot more data available about marginal values than about dependencies.  

In comparison with the existing models of imprecise Markov chains the most important differences are that probability models behind our model are not necessarily convex and that in general they do not satisfy Bellman's principle of optimality (see e. g. \cite{puterman2014markov}). Consequently calculating bounds for multiple steps transition probabilities is a much more computationally intensive task. 

We give the detailed description of the model in Sections~\ref{s-ms} and \ref{s-imc}. 

% Moreover and that they do not satisfy Bellman's principle of optimality (references).

\subsection{Results}
While our theoretical model is not very different from other models of imprecise Markov chains, there are substantial differences when it comes to computations. We will investigate computing $n$-steps transition probabilities, which are the basis for any analysis with Markov chains. As imprecision is involved, we cannot speak about single precisely given transition probabilities, but rather their lower and upper bounds. Moreover, in the case of imprecise probabilities, bounds for elementary events are not sufficient to specify the corresponding probability models. Therefore we have to consider computing bounds for more general expectations.

The existing models of imprecise Markov chains allow setting transitions from one state to others independently from one another. This ensures convexity of the underlying probability models and possibility to apply Bellman's principle of optimality. These properties then imply existence of a single local and therefore also global optimum, which is found by sequentially maximizing expectations via linear programming. Complexity of the problem thus remains linear in the number of time steps. The problem of finding extremal expectation in our settings becomes considerably more complicated. In general the problem is not convex and neither it satisfies Bellman's principle. Consequently, in general multiple local optima exist, and that backwards induction is not applicable. This means that irreducible dimensionality of the problem grows exponentially with the number of time steps. 

Our main numerical result is a local optimisation algorithm which we propose in  Section~\ref{s-nc}. Given an initial weight function it returns a local optimal solution. Global extrema, though, are still sought by taking various starting points and do local optimisation. As the size of the space of all feasible points is far to big to be tractable by any reasonable computer, we cannot provide a criterion that would definitely ensure that obtained solution is global maximum. But numerical testing shows that in most cases a reasonable approximation of global solution can be obtained by taking a moderate number of  starting points. Even more convincingly it shows that if weight functions were chosen by random, without applying the local optimization, then it would almost certainly take incomparably larger samples to get results comparably close to the optimal solution. While, as far as we are aware, no other algorithms exist for optimization of random walks on graphs with interval weights, we can only compare our method to random choice, which is therefore by far outperformed. 

\section{Model settings}\label{s-ms}
%\subsection{Random walk on graph with interval weights}
Let $\states$ be a non empty set of \emph{states}. We will usually denote the number of states by $s$. We consider random walks on the graph with vertices $\states$ and weighted edges that are given in the form of an interval \emph{weight function}. The probabilities of transitions between states are assumed to be proportional with the weights. More precisely, if $w\colon  \states^2 \to \RRpo$ is a weight function then 
\begin{equation}\label{eq-tprob}
P_w(x, y) = P_w(x_{n+1} = y|X_n = x) = \frac{w(x, y)}{\sum_{y\in \states} w(x, y)} = \frac{w(x, y)}{W(x)}. 
\end{equation}
In this paper we assume that the denominator in the above fraction is a fixed function of the state. That is, we assume a precise function $W\colon \states\to \RR_+$ as a sum of weights edges incident to a vertex. This restriction will enable us to obtain an efficient optimization algorithm. When modeling uncertainty we often have fairly good information on the long term distributions over the set of states (which are closely related to the corresponding weights) but much less certainty regarding the transition probabilities. We will thus allow weights $w(x, y)$ where $x\neq y$ to vary freely within given intervals and the remaining weight mass will be used to model the loop weight $w(x, x)$. 

Formally, we define a set 
\begin{equation*}
\wset = \left\{  w\colon  \states^2 \to \RRpo \colon \low w(x, y)\le  w(x, y) \le \up w(x, y), \sum_{y\in \states} w(x, y) = W(x) \right\}.
\end{equation*} 
where $\low w$ and $\up w$ are arbitrary such that 
\[ \low w \le \up w \qquad\text{and} \qquad \sum_{\substack{y\in \states\\ y\neq x }}  \up w(x, y) \le W(x). \] 
Every weight function in $\wset$ defines transition probabilities via equation \eqref{eq-tprob}. Our aim is to provide some basic properties of the corresponding Markov chains. 

Additionally, to avoid problems with uniqueness of the invariant distributions, we will assume that for all pairs of states 
\begin{equation}\label{eq-convention1}
\text{either} ~ \up w(x, y) = 0 ~\text{or}~ \low w(x, y)>0
\end{equation}
and that there is a path between every pair of states consisting of edges with strictly positive weights. 

%\subsection{The structure of the weights space}
%The set of all weight functions $\wset$ is clearly a convex set and as a set of $\frac{s(s+1)}{2}$-dimensional vectors it can be endowed with one of the equivalent finite dimensional vector norms. From now on let 
%\begin{equation*}
%\| w \| = \max_{(x, y)\in \states^2}{|w(x, y)|}
%\end{equation*}
%and the induced distance function is
%\begin{equation*}
%d(w, w') = \max_{(x, y)\in \states^2}{|w(x, y)-w'(x, y)|}.
%\end{equation*}
%We will call a weight function $w$ \emph{extremal} if 
%\[ w(x, y) \in \{ \low w(x, y), \up w(x, y) \} \]
%for every pair $(x, y)$ such that $x\neq y$. 

\section{Imprecise Markov chains}\label{s-imc}
\subsection{Transition operators with separately specified rows}
Markov chains whose parameters are only partially known have been studied recently into some details under the name \emph{imprecise Markov chains} (\cite{decooman-2008-a, skulj:09}) or \emph{Markov set chains} (\cite{hart:98}). Here we give basic ideas and notations related to the theory described in \cite{decooman-2008-a}.

An \emph{imprecise Markov chain} is a sequence $(X_n)_{n\in \NN\cup \{0\}}$ of random variables taking values in a finite state space $\states$. The imprecise distribution corresponding to some $X_n$ is given in the form of a set of probability distributions $\mathcal M_n$ consisting of distributions compatible with the given partial knowledge of the process. In the case where $\mathcal M_n$ are convex, they can be equivalently described in terms of lower expectation functionals
\begin{equation}\label{eq-lower-expectation}
\low E_n (f) = \min_{q\in \mathcal M_n} \sum_{x\in \states} q(x) f(x), 
\end{equation} 
where $q\colon \states\to \RR$ are probability mass functions corresponding to distributions in $\mathcal M$ and $f$ an arbitrary real valued map on $\states$.  

The transition law between states is also given in imprecise way. That is by assuming a set of transition operators $\mathcal T$, which is called an \emph{imprecise transition operator}. The following relation then holds 
\begin{equation}
\mathcal M_{n+1} = \mathcal M_n \mathcal T = \{ qT \colon q \in \mathcal M_n, T\in \mathcal T \}. 
\end{equation}
We do not assume transition probabilities being constant in time but only that they belong to the specified set of transition operators. Thus, an imprecise Markov chain is in principle time inhomogeneous with fixed constraints on transition probabilities. 

Moreover, an imprecise transition operator $\mathcal T$ maps some $f\in\RR^\states$ to a set $\mathcal Tf = \{ Tf \colon T\in \mathcal T\}$. An imprecise transition operator $\mathcal T$ is said to have \emph{separately specified rows} if for every $T, T'\in  \mathcal T$ there exists $\tilde T$ so that $\tilde T_{ij} = T_{ij}$ for every $j$ and $i\neq k$ and $\tilde T_{kj} = T'_{kj}$ for every $j$. That is the $k$th row can be chosen independently from the choice of other rows. 

Now we have the following important property. If an imprecise transition operator $\mathcal T$ has separately specified rows then there exist the minimal and maximal elements in the set $\mathcal Tf$ for every $f\in\RR^\states$, denoted by $\low Tf$ and $\up Tf$. The mappings $\low T\colon f\mapsto \low Tf$ and $\up T\colon f\mapsto \up Tf$ are called the lower and the upper transition operators respectively, and their values can be calculated via linear programming. Let  $\low E_0$ be an initial lower expectation operator and $\low T$ a lower transition operator. 
The expectation of $f(X_n)$, where $f\in \RR^\states$ is calculated by repeatedly applying $\low T$ using linear programming and finally apply $\low E_0$, again via linear programming: 
\begin{equation}\label{eq-n-step-expectation}
\low E_n (f) = \low E_0 (\low T^n f).
\end{equation}
The above equation generalizes calculation of $n$-step transition probabilities, since, for instance, 
\[ \low P(X_n = y| X_0 = x) = E_x(\low T^n 1_{\{y\}}) = \low T^n 1_{\{y\}}(x), \]
where $1_A$ denotes the indicator function of the set $A\subseteq \states$. While in the case of precise Markov chains $n$-step transition probabilities between single states completely determine the distribution of $X_n$, in the case of imprecise transitions expectations \eqref{eq-n-step-expectation} must be used instead (see \cite{decooman-2008-a} 
or \cite{skulj:09} for more details). 

\subsection{Transition operators on weighted interval graphs}
The transition operator with respect to a weight function $w$ is a map $T_w\colon \RR^\states\to \RR^\states$ that is defined with 
\begin{equation}\label{eq-trop}
T_wf(x) = \sum_{y\in\states}P_w(x, y) f(y) = \sum_{y\in\states}\frac{w(x, y)}{W(x)} f(y).
\end{equation}
%To a given set of weights $\wset$ the set of corresponding transition operators $\tset = \{ T_w\colon w\in \wset \}$ is assigned. Let $\mathbf{w} = (w_1, \ldots, w_n)$ be a vector of weight functions that rule the transition at consecutive steps. Then we define recursively 
%\begin{equation}
%	T_\mathbf{w}f(x) = T_{w_1} \ldots T_{w_n} f.
%\end{equation}
Similarly we can define the action of $T_w$ from the right by
\begin{equation}
qT_w(y)	= \sum_{x\in\states}q(x) P_w(x, y) = \sum_{x\in\states}q(x)\frac{w(x, y)}{W(x)}.
\end{equation}
In particular, if $q$ is a probability mass function corresponding to $X_n$ on the set of states then $qT_w$ is the probability mass function corresponding to $X_{n+1}$. 

We will stick to our general assumption that transition operators are (non-specified) function of time, rather than being constant in time. 
We now extend naturally both operators to vectors of weight functions $\mathbf{w} = (w_1, \ldots, w_n)$ with 
\begin{equation}
T_{\mathbf{w}}f = T_{w_1}\ldots T_{w_n}f
\end{equation}
and 
\begin{equation}
qT_{\mathbf{w}} = qT_{w_1}\ldots T_{w_n}.
\end{equation}
Given a set of weights $\wset$, we define an imprecise transition operator $\mathcal T = \{ T_w \colon w\in \wset \}$. It is clear that so defined imprecise transition operator does not possess the separately specified rows property. In fact there is no way to involve such a property, because of the symmetry of weights, that is the entry $w(x, y)$ which determines transition probability from $x$ to $y$ also determines the reverse transition probability. Yet those two belong to different rows of the corresponding transition matrix, namely, the first one in the row corresponding to $x$ and the second one to the row corresponding to $y$. As a consequence, the notion of an upper or lower transition operators does not make sense here, as there is no unique maximal or minimal elements in the set $\tset f := \{ T_w f\colon T_w\in \tset  \}$. Moreover, while sets of the form $\tset$ or $\tset f$ are convex, this is not any more the case with more general sets, such as $\tset^2 := \{ T_{w_1}T_{w_2}\colon w_1, w_2\in \tset \}$ or $\tset^2f$. This also means that optimization methods based on linear programming that are successfully applied in the theory of imprecise Markov chains cannot be applied on our case. 

%\subsection{Regularity and coefficients of ergodicity}
An imprecise Markov chain is said to be \emph{regular} if there exists some positive integer $r$ such that all transition operators in $\tset^n$, where $n\ge r$, have all elements positive. According to our convention \eqref{eq-convention1} every state is reachable from any other state, and since loops are also possible with strictly positive probability the chain is acyclic and therefore regular. It follows then (see \cite{2013:skulj-hable}) that there exists the unique invariant set $\mathcal M$ of probability distributions. Assuming fixed marginals, it follows that in our case the marginal distribution $\pi$, where $\pi(x) = \frac{W(x)}{W}$ is the common unique distribution corresponding to all operators in $\mathcal T$. Thus, $\{ \pi \}\mathcal T = \{ \pi \}$, which implies that $\{ \pi \}$ is the unique invariant set of distributions. 

%Knowing that a chain is regular tells us that it will eventually converge to equilibrium, but not how fast convergence will be. As measures of convergence \emph{coefficients of ergodicity} are often used. For a precise Markov chain, a coefficient of ergodicity is any function $\tau$ defined on the set of transition matrices with the properties:
%\begin{enumerate}[(i)]
%\item  $\tau(P_1P_2) \le \tau(P_1)\tau(P_2)$;
%\item $\tau (P) = 0 \text{ whenever $P$ has rank 1}.
%$\end{enumerate}

\subsection{Reversibility}
One of the most important properties of Markov chains that can be represented as random walks on graphs is that they are reversible processes. That means that we have equal probability to observe a sequence of states if their order is reversed:
\begin{equation}\label{eq-reversibility-rpr}
P(X_1 = x_1, \ldots, X_n = x_n) = P(X_1 = x_n, \ldots, X_n = x_1), 
\end{equation}
assuming that $(X_1 = x_1) = \pi(x_1)$, where $\pi$ is the unique invariant distribution. Applying the above property to the case of $n=2$ we obtain the \emph{detailed balance condition}: 
\begin{equation}
\pi(x)P(x, y) = \pi(y)P(y, x), 
\end{equation}
where $P(x, y)$ is the transition probability between $x$ and $y$. 

Clearly, a precise random walk on a graph with weight function $w$ satisfies the detailed balance condition, due to the fact that 
\[ P(X_1 = x, X_2 = y) = \frac{w(x, y)}{W} = P(X_1 = y, X_2 = x), \]
where $W$ denotes the sum of weights for all edges. The symmetry of interval weights also clearly implies that the lower probabilities $\low P(X_1 = x, X_2 = y)$ and $\low  P(X_1 = y, X_2 = x)$ are the same and equal to $\dfrac{\low w(x, y)}{W}$. 

Similarly we can calculate probabilities of the elementary events for more consecutive steps. Thus, for instance
\begin{equation}\label{eq-lpr-seq}
\low P(X_1 = x_1, \ldots, X_n = x_n) = \frac{\prod_{i = 1}^{n-1}\low w(x_i, x_{i+1})}{W\prod_{j=2}^{n-1}W(x_j)}	, 
\end{equation}
which again is equal to the lower probability of the reversed sequence of states. 

\section{Numerical calculations}\label{s-nc}
\subsection{Calculating expectations bounds}\label{ss-opt}
Although we have a very simple expression \eqref{eq-lpr-seq} that allows calculating the lower probabilities of sequences of states it cannot be used directly to find lower (or upper) probabilities of more general events. An example would be calculating the 2-step lower probability of transition from $x$ to $y$. In the precise case knowing the probabilities of all chains of states of length 3 would allow calculating such probability:
\[ P(X_2 = y| X_0 = x) = \sum_{z\in\states} P(X_2 = y|X_1 = z)P(X_1=z|X_0=x). \]
However, the above formula is incorrect if lower probability $\low P$ is taken instead of $P$. The reason is that lower probabilities are in general \emph{non-additive}. In our case, for instance, the lower probability $\low P(x, y) = P_w(x, y)$ for some particular weight function $w$, but $\low P(x, y') = P_{w'}(x, y')$ for another weight function $w'$, and there is usually no weight function that would induce the lower probabilities simultaneously. 

What we need to find in the case of minimizing the 2-step transition probabilities is 
\begin{equation}
\low P(X_2 = y | X_1 = x) = \min_{w_1, w_2\in \wset} \sum_{z\in\states} P_{w_1}(x, z)P_{w_2}(z, y).
\end{equation}
A more general problem is to find the bounds for the expectation of some function $f(X_n)$ given the information that $X_0 = x$, or some probability distribution of $X_0$ over the set of states. But even if this seems like an unnecessarily more complex problem, we would in fact not gain much in terms of simplicity by restricting to $n$-step transition probabilities alone. 

% Let $x, y\in \states$ be given. The lower bound for probability of transition from $x$ to $y$ in one step is the minimal value of the expression $\frac{w(x, y)}{W(x)}$, where minimum is taken over the set of all admissible weights, that is $\frac{\low w(x, y)}{W(x)}$. Although this is a correct result, it less informative as expected. If, for instance, we would look for the expectation of some function $f(X_{n+1})$ given the information that $X_n = x$, we would have to consider the whole set of possible probability distributions of $X_{n+1}$, not just the lower bounds. When proceeding one step further the situation becomes even more complicated, since we now have to consider the set of all possible distributions of $X_{n+2}$ conditional on all possible distributions of $X_{n+1}$. In general we would like to be able to calculate the distribution 

A (precise) probability distribution over $\states$ can be described via probability mass function (pmf) $q\colon \states \to \RRpo$ where $q(x) = P(X_0 = x)$. The expectation of some $f\in \RR^\states$ with respect to $q$ is the scalar product
\begin{equation}
\langle q, f \rangle := E_q(f) = \sum_{x\in\states} q(x) f(x) .
\end{equation}
To extend the above formulation for more general case of the expectation after $n$ steps, we consider the vector of weight functions $\mathbf{w} = (w_1, \ldots, w_n)$. Thus at $k$-th step the weight function $w_k$ is assumed to induce the transitions. We will then denote
\begin{equation}
\langle q, f \rangle^n_{\mathbf{w}}	:= \langle q, T_{\mathbf{w}} f \rangle. 
\end{equation}
Our goal is to find 
\begin{equation}\label{eq-general-bounds}
\low{\langle q, f \rangle}^n = \min_{\mathbf{w}\in \wset^n}\langle q, f \rangle^n_\mathbf{w} \qquad \text{and}\qquad  \up{\langle q, f \rangle}^n = \max_{\mathbf{w}\in \wset^n}\langle q, f \rangle^n_\mathbf{w}.
\end{equation}
Since, clearly, $\up{\langle q, f \rangle}^n = -\low{\langle q, -f \rangle}^n$, we only need to consider the minimization version. Allowing arbitrary real valued functions $q$ instead of restricting to probability mass functions does not change anything in the sense of problem complexity. Therefore we will from now on assume $q$ and $f$ to be arbitrary real valued functions. 
\begin{prop}
	Let $f$ and $q\in \RR^\states$ and $\mathbf{w}\in \wset^n$. Then $\langle qT_\mathbf{w}, f \rangle = \langle q, T_\mathbf{w}f \rangle$. 
\end{prop}
\begin{proof}
	We proceed by induction on the length $n$ of vector $\mathbf{w}$, where the case $n=0$ is trivial. Take now 
	\begin{align*}
	\langle qT_\mathbf{w}, f \rangle & = \langle qT_{w_1}, T_{(w_2, \ldots, w_{n})}f \rangle 
	\intertext{denote $\tilde f = T_{(w_2, \ldots, w_{n})}f$ and continue with }
	& = \sum_{y\in\states} qT_{w_1}(y) \tilde f(y) \\
	& = \sum_{y\in\states} \sum_{x\in\states} q(x)\frac{w_1(x, y)}{W(x)} \tilde f(y) \\
	& = \sum_{x\in\states} q(x) \sum_{y\in\states} \frac{w_1(x, y)}{W(x)} \tilde f(y) \\
	& = \langle q, T_{w_1} \tilde f \rangle  = \langle q, T_{\mathbf{w}} f \rangle.		 
	\end{align*}	
\end{proof}
Clearly the following holds. 
\begin{prop}
	The mapping $(q, w, f)\mapsto \langle q, f \rangle_w^1$ is linear in all variables. 
\end{prop}

The first step to finding bounds \eqref{eq-general-bounds} is to set $\mathbf{w} = (w)$, that is to take a single step, and find the weight function that makes the expectation $\langle q, f \rangle_w^1$ extremal. 
\begin{prop}[Optimality principle]\label{pr-op}
	Let $q, f\colon \states\to \RR$ and $w$ be some weight function. Set $h(x) = q(x)/W(x)$ for every $x\in \states$ and define 
	\[ \psi_{h, f}(x, y) = (h(x)-h(y))(f(y)-f(x)). \] 
	%	where $\mathrm{sign}(x)=\begin{cases}
	%	1 & \text{if }  x>0; \\
	%	-1 &\text{if } x<0; \\
	%	0 & \text{if } x = 0.
	%	\end{cases}$	
	Then $\langle q, f \rangle_w^1 = \low{\langle q, f \rangle}^1$ if 
	\[ 
	w(x, y) = \begin{cases}
	\low w(x, y) & \text{if } \psi_{h, f}(x, y) > 0; \\
	\up w(x, y) & \text{if } \psi_{h, f}(x, y) \le 0.
	\end{cases}
	\]
\end{prop}
\begin{proof}
	We prove the proposition by contradiction. Suppose that $\psi_{h, f}(x_0, y_0) > 0$ but $w(x_0, y_0)> \low w(x_0, y_0)$. Let $0 < d\le w(x_0, y_0) - \low w(x_0, y_0)$ and set 
	\[
	w'(x, y) = \begin{cases}
	w(x, y) & \text{if } (x, y) \not\in \{ (x_0, y_0), (x_0, x_0), (y_0, y_0), (y_0, x_0) \}; \\
	w(x, y) - d & \text{if } (x, y) \in \{ (x_0, y_0), (y_0, x_0) \}; \\
	w(x, y) + d & \text{if } (x, y) \in \{ (x_0, x_0), (y_0, y_0) \} .
	\end{cases}
	\]
	We have that 
	\begin{align*}
	\langle q, f \rangle_{w'}^1 & = \sum_{x\in\states} q(x) \sum_{y\in\states} \frac{w'(x, y)}{W(x)} \\
	& = \sum_{x\in\states} \sum_{y\in\states} w'(x, y) h(x) f(y) \\
	& = \langle q, f \rangle_{w}^1 + d (h(x_0)f(x_0) + h(y_0)f(y_0) - h(x_0)f(y_0) - h(y_0)f(x_0) ) \\
	& = \langle q, f \rangle_{w}^1 - d \psi_{h, f}(x_0, y_0) \\
	& < \langle q, f \rangle_{w}^1,
	\end{align*}
	which contradicts minimality of $\langle q, f \rangle_{w}^1$. 	
	
	The case where $\psi_{h, f}(x_0, y_0)<0$ is proved similarly.
\end{proof}

\begin{cor}\label{cor-ext-char}
	Let $q, f\in \RR^\states$ be arbitrary mappings. Then 
	\begin{equation}\label{eq-max-sp}
	\min_{\mathbf{w}\in \wset^n} \langle q, f \rangle^n_\mathbf{w} 
	\end{equation}
	is attained in a vector $\mathbf{w}=(w_1, \ldots, w_n)$, where all $w_i$ are extremal. 
\end{cor}
\begin{proof}
	Let $\mathbf{w} = (w_1, \ldots, w_n)$ be a vector minimizing \eqref{eq-max-sp}. We will show that then there exists a vector $\mathbf{w}'$  with extremal components such that $ \langle q, f \rangle^n_{\mathbf{w}'} \ge  \langle q, f \rangle^n_\mathbf{w}$. 
	
	Suppose that $i$th component $w_i$ of $\mathbf{w}$ is not extremal. Then set $\tilde q = qT_{(w_1, \ldots w_{i-1})}$ and $\tilde f = T_{(w_{i+1}, \ldots w_{n})}$. It easily follows from Proposition~\ref{pr-op} that the expression $\langle \tilde q, \tilde f \rangle_w^1$ is minimized by an extremal weight function $w$. Thus $\langle \tilde q, \tilde f \rangle_w^1 \ge \langle \tilde q, \tilde f \rangle_{w_i}^1 = \langle q, f \rangle^n_\mathbf{w}$. Thus replacing $w_i$ with $w$ in $\mathbf{w}$ would still maximize \eqref{eq-max-sp}. By repeating the same argument for all components with non-extremal weight functions, we confirm the corollary. 
\end{proof}

\subsection{The structure of the set of extremal weights}\label{ss-struct-extremal}
The problem of calculating the extreme value and the optimal weights vector in the expression  \eqref{eq-max-sp} is a high dimensional optimization problem. We need to find the optimal value of an $nm$ dimensional real-valued vector where $n$ stands for the number of time steps and $m$ for the number of edges of the graph. Knowing that all feasible vectors of weights where the extremal values may be found is restricted to the extremal weight vectors brings the number of the values to be considered down to the set of all $nm$ dimensional binary vectors. That gives $2^{mn}$ possibilities, which is in general still far more than a number tractable on any computer system. Moreover, it is not obvious how to endow this space with such a metric or norm structure that would help with the optimization. Therefore the aim of the rest of this paper is to give a heuristic approach to finding reasonable approximations of the target extreme values and insights into the practical complexity of the problem, which is apparently too hard to be tackled exactly. 

Yet there is an additional reduction of the complexity, which follows from the theoretical results in Section~\ref{ss-opt}. Every component $w_i$ of an extremal weight vector $\mathbf{w}$ must be of the form $w_\psi$, where $\psi$ is some function of the form $\psi_{h, f}$. In fact the sign of $\psi_{h, f}$ only matters. Clearly the sign of $\psi_{h, f}$ depends on the ordering of states induced by the maps $h$ and $f$. In other words, every pair of orderings of states induces a feasible extremal weight function. That is, for every component we need to consider at most $(s!)^2$ extremal weight functions instead of $2^{m}$ of all possible extremal weight functions, where $s$ denotes the number of states and $m$ the number of edges in the graph. The number of feasible weights additionally decreases when taking into account the orderings of $q$ and $f$, but still remains much too large to allow any kind of exhaustive examination. Therefore a reasonable approach is to start with some initial weight vector and try to improve it repeatedly until finding an optimal solution. Though, as we show in the continuation this process does not necessarily lead to a global optimum. 

\subsection{Finding local extrema}
%The following proposition now characterizes local extrema. 
%\begin{prop}
%	Let $q, f\colon \states \to \RR$ be arbitrary mappings and let $n\in \NN$ be given. Then define the map $ \Sigma\colon \wset^n \to \RR$ with $\Sigma(\mathbf{w}) = \langle q, f \rangle^n_\mathbf{w}$. The vector $\mathbf{w}_0$ is a local extreme for $\Sigma$ if $w_i$ maximizes $\langle \tilde q, \tilde f \rangle_w$ for every $0 \le i \le n$, where $\tilde q = qT_{(w_1, \ldots, w_{i-1})}$ and $\tilde f = T_{(w_{i+1}, \ldots w_{n})}$. 
%\end{prop}
%\begin{proof}
%	Clearly $\mathbf{w}$ cannot be a local extreme if it does not satisfy conditions from the proposition. 
%	
%	Ideja: w, w' ekavivalentna če \psi_i enak predznak na vseh. Pokazati: \psi je zvezna funkcija f. Za male spremembe w ne spremeni predznaka. Zato sprememba w pomeni smer v manj optimalno vrednost. 
%\end{proof}

Corollary~\ref{cor-ext-char} gives a necessary condition for $\mathbf{w}$ to give extremal value of $\langle q, f \rangle^n_\mathbf{w}$. However, there may be several weight vectors satisfying this condition, yielding different values. That is, we may have multiple local extrema due to non-convex nature of the problem. Let us illustrate this with an example. 
\begin{ex}
	Let $\states = \{ 1, 2 \}$, 
	\[ \low w = \begin{bmatrix}
	0.1 & 0.2 \\
	0.2 & 0.1 
	\end{bmatrix}, \up w = \begin{bmatrix}
	0.8 & 0.9 \\
	0.9 & 0.8 
	\end{bmatrix}, \]
	$q = (1, 0)$, and $f = (0, 1)$. The marginal weight function is then $W = (1, 1)$. Our goal is to minimize $qT_{w_1}T_{w_2}f$, which in our case corresponds to the lower two step transition probability $\low P(X_2 = 2|X_0 = 1)$.

	Take
	\[
	w = \begin{bmatrix}
	0.1 & 0.9 \\
	0.9 & 0.1 
	\end{bmatrix}, w' = \begin{bmatrix}
	0.8 & 0.2 \\
	0.2 & 0.8 
	\end{bmatrix},
	\]
	which are the only extremal weight functions. The transition operators corresponding to $w$ and $w'$ coincide with $w$ and $w'$ because of both marginal weights being equal to 1. 
	
	Now we have that $f_2 := T_{w}f = (0.9, 0.1)$ and $f'_2 := T_{w'}f = (0.2, 0.8)$. Hence 
	\[
	\mathrm{sign}(\psi_{q, f_2}) = \begin{bmatrix}
	0 & 1 \\
	1 & 0 
	\end{bmatrix} \qquad \text{and} \qquad 
	\mathrm{sign}(\psi_{q, f'_2}) = \begin{bmatrix}
	0 & 0 \\
	0 & 0 
	\end{bmatrix}
	\]
	Hence the expression $qT_{w_1}f_2$ is minimized by taking $w_1 = w$ and similarly the expression $qT_{w'_1}f'_2$ is minimized by taking $w'_1 = w'$. 
	
	To make things clearer we will parametrize all two-step weight vectors. Let 
	\[
	w_1(\alpha) = \alpha w + (1-\alpha) w' \qquad\text{and} \qquad w_2(\beta) = \beta w + (1-\beta) w'.
	\]
	Now we can express explicitly 
	\[ F(\alpha, \beta) = qT_{w_1(\alpha)}T_{w_2(\beta)}f = 0.32 + 0.42 \alpha + 0.42\beta - 0.98 \alpha \beta, \]
	whose local extrema in $[0, 1]^2$ are $F(0, 0) = 0.32$ and $F(1, 1)=0.18$ which are local minima while in the local maxima in $(0, 1)$ and $(1, 0)$ the same value 0.74 is attained. The only stationary point in the interior of $[0, 1]^2$ is a saddle point and thus not a local extreme. 	
\end{ex}
Having characterized local extrema, we can now provide a simple Algorithm~\ref{alg-lext} that finds local extrema from some starting weight vector $\mathbf{w}$. 
\begin{algorithm}
	\caption{FindLocalMinimum}\label{alg-lext}
	\begin{algorithmic}[1]
		\Function{LocalMinimum}{$q, initialWeightVector, f$} 
		\State	$weightVector \gets initialWeightVector$
		\Repeat
		\State $k \gets setSplitPoint$   \Comment{set the division point} \label{alg-code-set-split}
		\State $wl \gets weightVector(1:k-1)$ \Comment{left part of the weight vector}
		\State $wr \gets weightVector(k+1:n)$ \Comment{right part of the weight vector}
		\State $ql \gets qT_{wl}$
		\State $fr \gets T_{wr}f$
		\State $wnew \gets w_{\psi_{ql, fr}}$ \Comment{set the weight function $w$ that minimizes $qlT_wfr$}
		\State $weightVector \gets (wl, wnew, wr)$
		\Until there are no more division points where $qlTw_kfr$ is not minimal
		\EndFunction
	\end{algorithmic}
\end{algorithm}
We have left the way how the split point is selected open on purpose in Algorithm~\ref{alg-lext}, ~line~\ref{alg-code-set-split}, because there are many ways how we can proceed with this. Thus, we can for instance start from left to right and repeat the process until we find no more locally non-optimal weight functions. Or we might go the other way around. The order does affect the results. That is, not only that the number of iterations needed may be different, but also the resulting locally optimal weights vector $\mathbf{w}$ may be different. 

\subsection{Testing the local algorithm}
It takes just a slight modification of Algorithm~\ref{alg-lext} to find local maxima instead of minima. To understand its efficiency for finding global extrema, we implemented some numerical testing with random interval weights in graphs of various sizes to answer the following relevant questions:
\begin{enumerate}
	\item How many unique local extrema does typical problem of the form \eqref{eq-general-bounds} have?
	\item How does the order of split points affect the resulting local extreme?
	\item Can the value of $\langle q, f \rangle_{w_\text{initial}}^n$ in any way predict the value in the resulting local extreme?
	\item Does the value in a local extreme affect the likelihood of the algorithm resulting in that extreme?
\end{enumerate}
We have tested the algorithm on graphs with 4, 6 and 8 vertices. Lower weights $\low w(x, y)$ were randomly generated using exponential distribution with expected value $0.8$ and the upper bounds were generated as $\up w(x, y) = \low w(x, y)X$ where $X$ were random numbers distributed exponentially with the expected value 1. Entries of both $q$ and $f$ were generated as random numbers distributed exponentially with expected value 1.5. Each graph contained about 1/4 of pairs of vertices that were not connected. 

For each set of parameters, i.e. $\low w, \up w, q$ and $f$, a sample of 1500 extremal weight vectors was generated  using random permutations (see the last part of Section~\ref{ss-struct-extremal}) and then calculated the corresponding local extreme using Algorithm~\ref{alg-lext}. We tested for random walks of lengths 2, 4 and 6 respectively. For each size of graphs and random walk length a sample of 200 sets of parameters was generated. 

Typical distribution of the number of (discovered) extreme points is best modelled with exponential distribution with the means that depend on the number of vertices and time steps. The parameters are listed in Table~\ref{tb-lextrema}.
\begin{table}
	\centering
	\begin{tabular}{rr|rrr}
		\toprule 
		& & \multicolumn{3}{|c}{time steps} \\ 
		& & 2 & 4 & 6 \\ \hline 
		\multirow{3}{*}{vertices} & 4&  1.9 & 13.4 & 80.6 \\
		& 6 & 3.2 & 46.8 & 251.2 \\
		& 8 & 5.3 & 100.3 & 411.4 \\
		\bottomrule
	\end{tabular}		
	\caption{Average number of local extrema.}\label{tb-lextrema}
\end{table}
Further we tested the influence of the order in which non-optimal weight functions are selected and optimized in Algorithm~\ref{alg-lext},~line~\ref{alg-code-set-split}. We tested left-to-right and right-to-left order. It turns out that most of the time the resulting local extrema do not coincide. However, the comparison of overall frequencies of the obtained local extrema do not show any systematical differences in their distributions. In Figure~\ref{fig:L-R-comparison} frequencies of local minima and maxima respectively are depicted for the left-to-right and right-to-left orders. One can observe similar distributions for both orders. 
\begin{figure}[h!]
	\centering
	\includegraphics[scale=0.7]{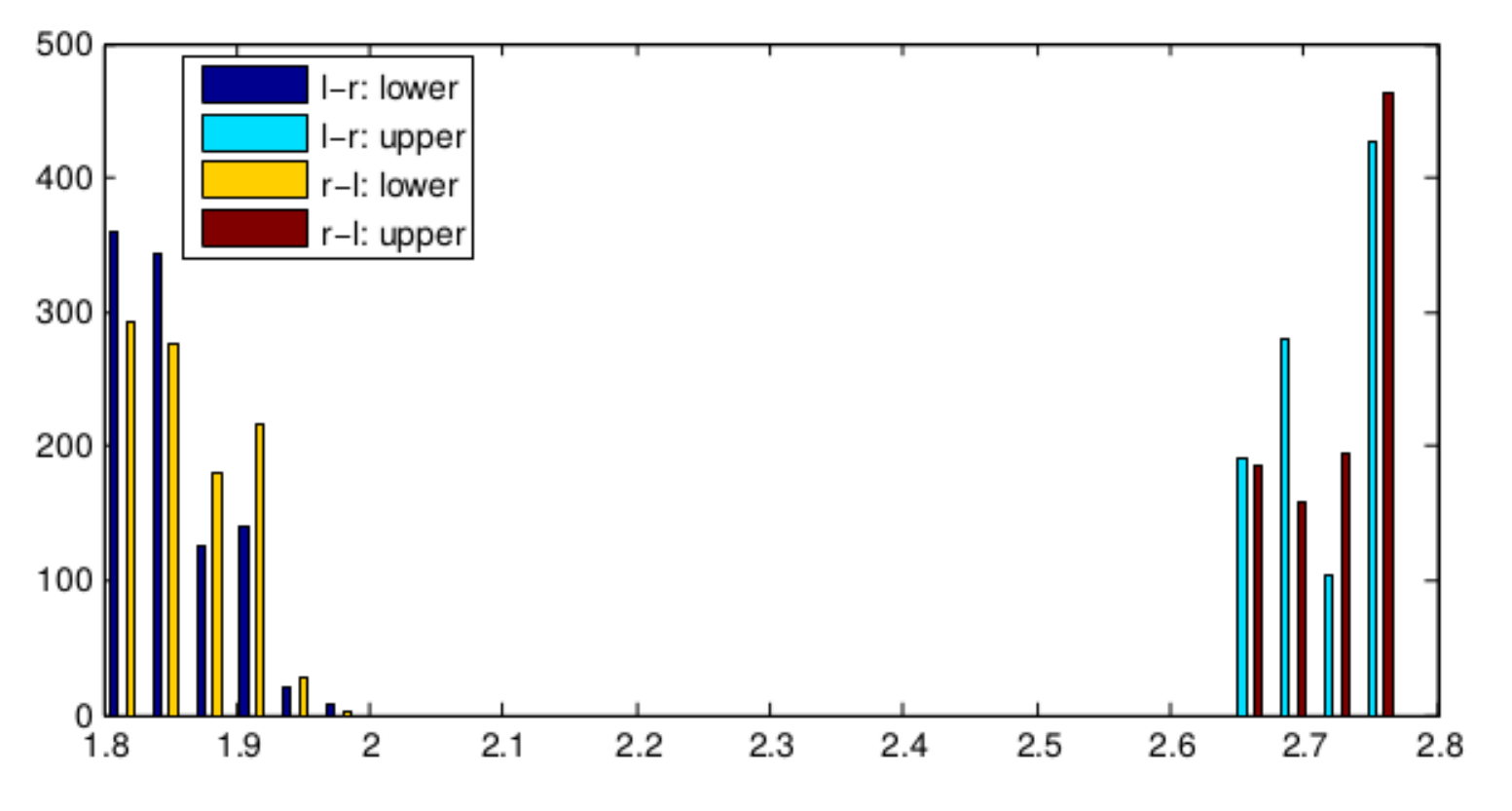}
	\caption{Frequencies of left-to-right and right-to-left order in Algorithm~\ref{alg-lext} for lower and upper bounds.}
	\label{fig:L-R-comparison}
\end{figure}

One might expect that the starting weight function can give some information about the local extreme it leads to. It could be for instance that low value in the starting point predicts lower value in the resulting local minimum, which could be used to do preliminary selection of the starting points. However, empirical results show no significant correlation between the two values. A graph showing dependency between the initial and optimized value is shown in Figure~\ref{fig:initial-optimized-plot}. It clearly shows that the initial value does not reveal much information about the corresponding local maximum. Although the efficiency of the algorithm can be clearly observed by comparing the initial and optimized values. 
\begin{figure}[h!]
	\centering
	\includegraphics[scale=0.5]{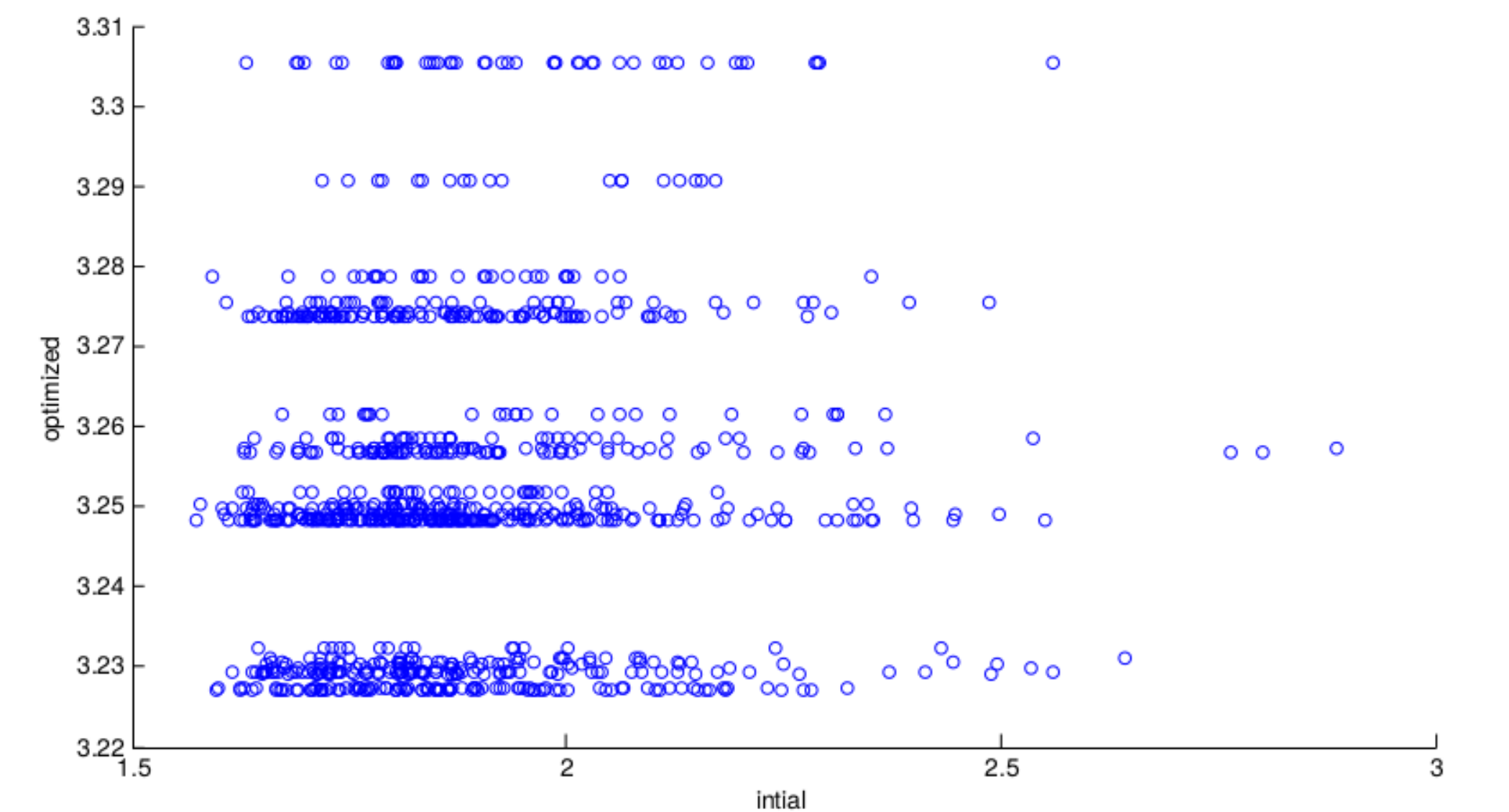}
	\caption{Comparison of the initial value of the maximizing function and the value after applying the local optimization algorithm.}
	\label{fig:initial-optimized-plot}
\end{figure}
Experimental results neither show any particular general pattern that would suggest that the value of a local extreme would impact the probability to be attained by the local optimization algorithm. 

\subsection{From local to global extrema}
So far the best way to find global extrema for random walks in weighted graphs is to take some sample of random initial weight vectors, do the local optimization, and hope that one of the so obtained local extrema is a global extreme. This method, however, does not contain any decisive criterion whether the best obtained solution is globally best solution. It is therefore not clear how big sample one has to take in order to get an estimate reasonably close to the true optimal solution with reasonable certainty. In most cases not too large samples are needed for this. 
\begin{figure}[h]
	\centering
	\includegraphics[scale=0.7]{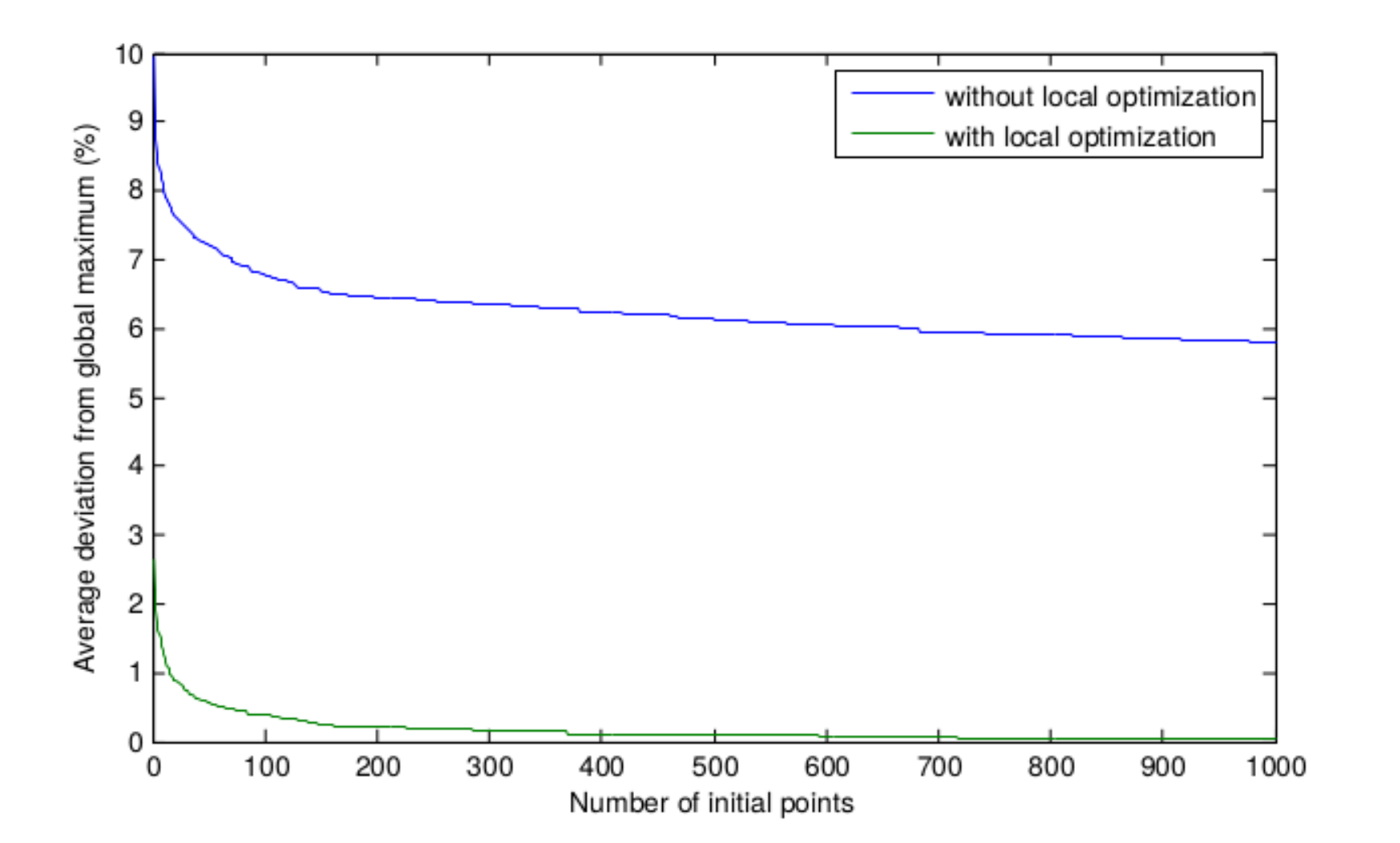}
	\caption{Average relative deviations from optimal solution for given sample of initial points with and without local optimization.}
	\label{fig:avg-deviations-samples}
\end{figure}
\begin{figure}[h!]
	\centering
	\includegraphics[scale=0.7]{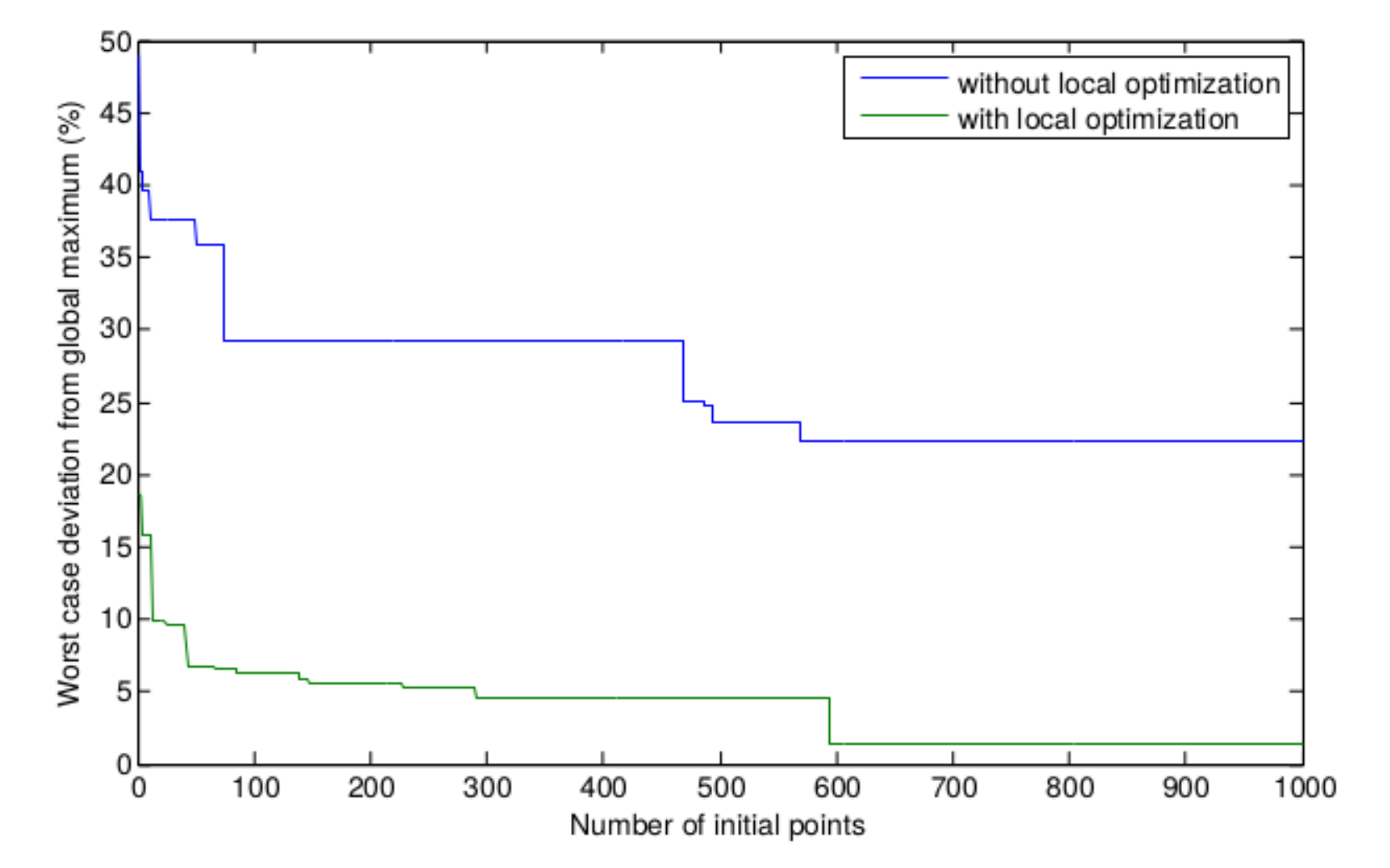}
	\caption{Maximal relative deviations from optimal solution for given sample of initial points with and without local optimization.}
	\label{fig:max-deviations-samples}
\end{figure}
We have generated a sample of 300 parameter sets for interval weighted graphs with 8 vertices  and the same number of time steps. All $q$ and $f$ were non-negative to ensure non-negativity of the resulting optima and therefore comparison of relative deviations. Both interval weights and widths were generated as exponentially distributed random values. For each parameter set a sample of 2500 randomly generated initial points was generated and then optimized to get the corresponding local extrema. We have calculated the values $\langle q, f\rangle^n_\mathbf{w}$ for the initial extremal vector of weights $\mathbf{w}$ and for the optimized weight vector $\mathbf{w}_{\mathrm{opt}}$. In Figure~\ref{fig:avg-deviations-samples} the average relative deviation, in percentage of the best solution, from the most optimal value found is graphed for initial and optimized vector weights depending on the size of the sample. In Figure~\ref{fig:max-deviations-samples} the worst case (maximal) relative deviation is graphed depending on the sample size. It can be clearly observed that locally optimized solutions by far outperform best solutions that might be found by randomly generated extremal weights only. Secondly, we can also observe that even if it cannot be guaranteed that best solution has been found among locally optimized solutions, the deviations become reasonable at not too big sample sizes.

\section{Conclusion and further work}
As far as we are aware this paper is a first attempt to model random walks on weighted graphs with interval weights, and also reversible imprecise Markov chains. We have addressed the most basic question about calculating transition probabilities for multiple time steps. This is a non-convex problem whose computational complexity grows exponentially with the number of time steps. Thanks to the local optimization algorithm we can find reasonable approximations of global optima with a tractable amount of computation. 

Our approach only works in the case where marginals are known precisely. It is therefore natural to try to extend calculations without the assumption of fixed sum of weights and even to allow more general sets of weight functions besides those which are described in terms of intervals. While both possibilities would be plausible, computational complexity may be an obstacle to prevent their efficient analysis. 

Regarding long term distribution our model is a simple one with known limit distribution. Although there are still several relevant questions to be addressed, such as computing mixing times, cover and hitting times. 

\bibliography{refer.bib}

\end{document}